\documentclass[twoside,leqno,10pt, A4]{amsart}
\usepackage{amsfonts}
\usepackage{amsmath}
\usepackage{amscd}
\usepackage{amssymb}
\usepackage{amsthm}
\usepackage{amsrefs}
\usepackage{latexsym}
\usepackage{mathrsfs}
\usepackage{bbm}
\usepackage{enumerate}
\usepackage{color}
\begin{document}

\newtheorem{theorem}[subsection]{Theorem}
\newtheorem{proposition}[subsection]{Proposition}
\newtheorem{lemma}[subsection]{Lemma}
\newtheorem{corollary}[subsection]{Corollary}
\newtheorem{conjecture}[subsection]{Conjecture}
\newtheorem{prop}[subsection]{Proposition}
\numberwithin{equation}{section}
\newcommand{\mr}{\ensuremath{\mathbb R}}
\newcommand{\mc}{\ensuremath{\mathbb C}}
\newcommand{\dif}{\mathrm{d}}
\newcommand{\intz}{\mathbb{Z}}
\newcommand{\ratq}{\mathbb{Q}}
\newcommand{\natn}{\mathbb{N}}
\newcommand{\comc}{\mathbb{C}}
\newcommand{\rear}{\mathbb{R}}
\newcommand{\prip}{\mathbb{P}}
\newcommand{\uph}{\mathbb{H}}
\newcommand{\fief}{\mathbb{F}}
\newcommand{\majorarc}{\mathfrak{M}}
\newcommand{\minorarc}{\mathfrak{m}}
\newcommand{\sings}{\mathfrak{S}}
\newcommand{\fA}{\ensuremath{\mathfrak A}}
\newcommand{\mn}{\ensuremath{\mathbb N}}
\newcommand{\mq}{\ensuremath{\mathbb Q}}
\newcommand{\half}{\tfrac{1}{2}}
\newcommand{\f}{f\times \chi}
\newcommand{\summ}{\mathop{{\sum}^{\star}}}
\newcommand{\chiq}{\chi \bmod q}
\newcommand{\chidb}{\chi \bmod db}
\newcommand{\chid}{\chi \bmod d}
\newcommand{\sym}{\text{sym}^2}
\newcommand{\hhalf}{\tfrac{1}{2}}
\newcommand{\sumstar}{\sideset{}{^*}\sum}
\newcommand{\sumprime}{\sideset{}{'}\sum}
\newcommand{\sumprimeprime}{\sideset{}{''}\sum}
\newcommand{\shortmod}{\ensuremath{\negthickspace \negthickspace \negthickspace \pmod}}
\newcommand{\V}{V\left(\frac{nm}{q^2}\right)}
\newcommand{\sumi}{\mathop{{\sum}^{\dagger}}}
\newcommand{\mz}{\ensuremath{\mathbb Z}}
\newcommand{\leg}[2]{\left(\frac{#1}{#2}\right)}
\newcommand{\muK}{\mu_{\omega}}

\title[Weighted First Moments of some special quadratic Dirichlet $L$-functions]{Weighted first moments of some special quadratic Dirichlet $L$-functions}

\date{\today}
\author{Peng Gao and Liangyi Zhao}

\begin{abstract}
In this paper, we obtain asymptotic formulas for weighted first moments of central values of families of primitive quadratic Dirichlet $L$-functions whose conductors comprise only primes that split in a given quadratic number field.  We then deduce a non-vanishing result of these $L$-functions at the point $s=1/2$.
\end{abstract}

\maketitle



\section{Introduction}

It is a conjecture due to S. Chowla \cite{chow} that a Dirichlet $L$-function is never zero at the central point $s=1/2$.  One way to address this problem is by studying the moments of central values of $L$-functions.  For the family of quadratic Dirichlet $L$-functions, M. Jutila \cite{Jutila} obtained the first and second moments of $L(1/2, \chi_d)$ with $\chi_d$ being the Kronecker symbol.  The error term in the asymptotic formula for the first moment in \cite{Jutila} was later improved in \cites{DoHo, MPY, ViTa}. For the second and third moment of this quadratic family, K. Soundararajan obtained asymptotic formulas with power savings in \cite{sound1}. The error term for the third moment was improved by A. Diaconu, D. Goldfeld and J. Hoffstein \cite{DGH} and later further improved by M. P. Young \cite{MPY2}.   More recently, an explicit lower order term in the third moment was found in \cite{DiWh} and under the assumption of the generalized Riemann hypothesis for Dirichlet $L$-function, an asymptotic formula for the fourth moment was proved in \cite{Shen}.  For families of Dirichlet $L$-functions associated with characters of higher orders, we note that S. Baier and M. P. Young studied the first and second moments of $L(1/2, \chi)$ for cubic Dirichlet $L$-functions in \cite{B&Y}.    With the knowledge of these moments, one can deduce, in manners not unlike the proof of Corollary~\ref{nonvan}, results on the non-vanishing of the $L$-functions under consideration. \newline

   In this paper, we study the first moments of central values of certain subfamilies of quadratic Dirichlet $L$-functions. Our result is motivated by the class field theory, which implies that when a number field is Galois over $\mq$, then the set of prime numbers in $\mq$ that split completely in it determines the number field uniquely.  For this reason, it is interesting to study the families of primitive quadratic Dirichlet $L$-functions whose conductors comprise only primes that split in a given number field. \newline

    We now let $K$ be a quadratic number field and let $S(K)$ be the set of odd rational integers that comprises only prime factors that split completely in $K$, i.e.
\begin{align*}
   S(K) & =\{q \in \mz : (q,2)=1, p|q \Rightarrow p  \text { splits completely in } K \}.
\end{align*}
  For technical reasons, other than a smooth weight, we consider the average of the central values of $L$-functions with an extra weight which essentially measures the number of distinct rational prime factors of the conductor of a given character. For $q \in \mz$, let $\omega(q)$ be the number of distinct rational prime factors of $q$.  Then, our result is
\begin{theorem}
\label{firstmoment}
  Let $K$ be a quadratic number field and let $w: (0,\infty) \rightarrow \mr$ be a smooth, compactly supported function.  Then for any $\varepsilon>0$,
\begin{equation}
\label{eq:1}
 \sum_{q \in S(K)}\;  \sumstar_{\substack{\chi \bmod{q} \\ \chi^2 = \chi_0}} 2^{\omega(q)} L(1/2, \chi) w\leg{q}{Q} =  Q P_K(\log Q) +O \left( Q^{1-\delta_0+ \varepsilon}+Q^{3/4 + \varepsilon} \right),
\end{equation}
where $P_K(x)$ is a linear function whose coefficients depend only on $K$ and $w$ (see \eqref{const} below for the expression for $P_K(x)$), $\delta_0$ is the currently best known constant in the subconvexity bound for a degree two $L$-function over $\mq$ (see \eqref{Lbound} below).  Here the $*$ on the sum over $\chi$ restricts the sum to primitive characters and the implicit constant in the error term depends on $K$, $w$ and $\varepsilon$.
\end{theorem}

As alluded to earlier, one can readily deduce the following non-vanishing result from the above theorem.

\begin{corollary} \label{nonvan}
Let $Q \in \natn$ and sufficiently large.  We have
\[ \# \left\{ \chi : \chi^2 = \chi_0, \; \chi \; \mbox{with conductor} \; q, \; q \in S(K) \cap [1, Q], \; L (1/2, \chi) \neq 0 \right\} \gg \frac{Q}{\log^{17} Q} . \]
\end{corollary}

\begin{proof}
It is well-known that a primitive quadratic Dirichlet character with odd conductor $q$ coincides with the Jacobi symbol modulo $q$ and $q$ must be square-free (see \cite[p.40]{Da}).  Note also that $2^{\omega(q)} = \tau(q)$ if $q$ is square-free.  Here $\tau(q)$ is the divisor function.  Therefore, using Theorem~\ref{firstmoment} and H\"older's inequality and choose a smooth function $w$ with support in $[0,1]$, we get
\[ Q \log Q \ll \left( \sum_{\substack{ q \in S(K) \\ q \leq Q}} \sumstar_{\substack{\chi \bmod{q} \\ \chi^2 = \chi_0 \\ L(1/2,\chi) \neq 0}} 1 \right)^{1/4} \left( \sum_{q \leq Q} \ \sumstar_{\substack{\chi \bmod{q} \\ \chi^2 = \chi_0}} \tau^4(q) \right)^{1/4} \left( \sum_{q \leq Q} \ \sumstar_{\substack{\chi \bmod{q} \\ \chi^2 = \chi_0}} L^2(1/2,\chi) \right)^{1/2}  . \]
The second factor above is $O(Q^{1/4} \log^{15/4} Q)$(see \cite[(1.80)]{iwakow}).  Using Theorem 2 of \cite{Jutila}, the third factor is $O(Q^{1/2} \log^{3/2} Q)$.  The corollary follows from these estimates.
\end{proof}

Before presenting the proof of Theorem~\ref{firstmoment}, we will give a brief summary of it.  We start by applying the approximate functional equation to the left-hand side of \eqref{eq:1}.  Those Dirichlet characters under consideration can be identified with Hecke characters in $K$ (see Lemma~\ref{lemma:quarticclass}).  Using Mellin inversion, we are led to study sums involving Hecke $L$-functions.  Moving the contour to the left, we will, in certain cases, encounter some poles whose residues will give rise to the main term in \eqref{eq:1}.  The remaining terms can all be estimated to give admissible error terms.  Among other things, a subconvexity bound for Hecke $L$-functions is needed in the analysis.

\subsection{Notations} The following notations and conventions are used throughout the paper.\\
\noindent  $K $ denotes a quadratic number field. \newline
\noindent  $\mathcal{O}_K $ denotes the ring of integers in $K$. \newline
\noindent  $D_K $ denotes the discriminant of $K$. \newline
$f =O(g)$, $f \ll g$ or $g\gg f$ means $|f| \leq cg$ for some unspecified positive constant $c$. \newline
$\mu_{K}$ denotes the M\"obius function on $K$. \newline
$\zeta_{K}(s)$ is the Dedekind zeta function for $K$.

\section{Preliminaries}
\label{sec 2}

In this section, we enumerate the tools used throughout the paper.
\subsection{Quadratic symbol and primitive quartic Dirichlet characters}
\label{sec2.4}
   For any prime ideal $\mathfrak{p} \subset  \mathcal{O}_K$ which is co-prime to $(2)$, we define for
   $a \in \mathcal{O}_K$, $(a, \mathfrak{p})=1$ by $\leg{a}{\mathfrak{p}}_K \equiv
a^{(N(\mathfrak{p})-1)/2} \pmod{\mathfrak{p}}$, with $\leg{a}{\mathfrak{p}}_K \in \{ \pm 1 \}$. When
$\mathfrak{p} | (a)$, we set
$\leg{a}{\mathfrak{p}}_K =0$.  Then this symbol, $\leg {\cdot}{\mathcal{A} }_K$, can be extended multiplicatively
to any ideal $\mathcal{A} \subset  \mathcal{O}_K$ with $(\mathcal{A}, 2)=1$ and is called the quadratic residue symbol in $K$. \newline

    For any $m \in \mz$ and any ideal $\mathcal{A} \subset  \mathcal{O}_K$ with $(\mathcal{A}, 2D_K)=1$, it follows from \cite[Proposition 4.2]{Lemmermeyer} that
\begin{align}
\label{2.1}
    \leg {m}{\mathcal{A}}_K=\leg {m}{N(\mathcal{A}) }_{\mq},
\end{align}
    where $\leg {\cdot}{\cdot}_{\mq}$ denotes the Jacobi symbol in $\mq$. \newline

    In particular, if $p$ is an odd prime in $\mq$ that splits completely in $K$ and $\mathfrak{p}$ is a prime ideal in $\mathcal{O}_K$ lying above $(p)$, then for any $m \in \mz$,
\begin{align}
\label{2.2}
    \leg {m}{\mathfrak{p} }_K=\leg {m}{p}_{\mq}.
\end{align}

  When $K$ is a quadratic number field, it is well-known from algebraic number theory (see \cite[pp. 111, 117]{F&A}) that a prime ideal $(p)$ in $\mz$ can either ramify, split (completely) or stay inert in $\mathcal{O}_K$. Moreover, a prime ideal $p$ in $\mz$ ramifies in $K$ if and only if $p$ divides $D_K$ (see \cite[Theorem 22]{F&A}). It follows from this and \eqref{2.2} that we have the following classification of all the primitive quadratic Dirichlet characters of conductor $q \in S(K)$:
\begin{lemma}
\label{lemma:quarticclass}
Primitive quadratic Dirichlet characters of conductor $q \in S(K)$ are of the form $\chi_{\mathcal{A}}: m \rightarrow (\frac{m}{\mathcal{A}})$ for some ideal $\mathcal{A} \subset \mathcal{O}_K $, $\mathcal{A}$ square-free, co-prime to $2D_K$ and not divisible by any rational primes, with norm $N(\mathcal{A}) = q$. Moreover, there are precisely $2^{\omega(q)}$ different ideals in $\mathcal{O}_K$ satisfying the above conditions that give rise to the same Dirichlet character.
\end{lemma}

\subsection{The approximate functional equation}
   We have the following approximate functional equation from \cite[Theorem 5.3]{iwakow}):
\begin{prop}
\label{prop:AFE}
Let $\chi$ be a primitive Dirichlet character $\chi$ of conductor $q$. For any $\alpha \in \mc, j \in \{ \pm 1 \}$, let
\begin{equation*}
 a_j=\frac {1-j}2, \quad  \epsilon(\chi) = i^{-a_{\chi(-1)}} q^{-1/2} \tau(\chi),
\end{equation*}
where $\tau(\chi)$ is the Gauss sum associated with $\chi$.  We define
\begin{equation} \label{vdef}
  V_{j}(x) = \frac{1}{2\pi i} \int\limits_{(2)}\gamma_{j}(s) x^{-s} \frac {\dif s}{s}, \quad \text{where} \quad \gamma_{j}(s) = \pi^{-s/2} \frac{\Gamma\left(\tfrac{1/2 + a_j+ s}{2}\right)}{\Gamma\left(\tfrac{1/2 + a_j}{2}\right)}.
\end{equation}
  Furthermore, let $A$ and $B$ be positive real numbers such that $AB = q$.  Then we have
\begin{equation} \label{approxfunc}
L(1/2, \chi) = \sum_{m=1}^{\infty} \frac{\chi(m)}{m^{1/2}} V_{\chi(-1)}\left(\frac{m}{A}\right) + \epsilon(\chi)\sum_{m=1}^{\infty} \frac{\overline{\chi}(m)}{m^{1/2}} V_{\chi(-1)}\left(\frac{m}{B}\right).
\end{equation}
\end{prop}

Note that (see \cite[Lemma 2.1]{sound1}) $V_{\pm 1}(\xi)$ are real-valued and smooth on $[0, \infty)$ and for the $l$-th derivative of $V_{\pm 1}(\xi)$, we have
\begin{equation} \label{2.07}
      V_{\pm 1}\left (\xi \right) = 1+O(\xi^{1/2-\epsilon}) \; \mbox{for} \; 0<\xi<1   \quad \mbox{and} \quad V^{(l)}_{\pm 1}\left (\xi \right) =O(e^{-\xi}) \; \mbox{for} \; \xi >0, \; l \geq 0.
\end{equation}
We remark here that the estimates in \eqref{2.07} are only proved for $V_{+1}$ in \cite{sound1} and the proof for $V_{-1}$ is similar and one get the same bounds with different implied constants.

\section{Proof of Theorem \ref{firstmoment}}
\label{sec3}

  We let
\begin{equation*}
\mathcal{M} := \sum_{q \in S(K)}\;  \sumstar_{\substack{\chi \bmod{q} \\ \chi^2 = \chi_0}} 2^{\omega(q)}  L(1/2, \chi) w\leg{q}{Q}.
\end{equation*}

   Applying the approximate functional equation \eqref{approxfunc} with $A=B = \sqrt{q}$ gives, noting that it follows from \cite[Chap. 2]{Da} that $\epsilon(\chi)=1$ when $\chi$ is quadratic,
\begin{align*}
 \mathcal{M} = 2 \sum_{q \in S(K)}\;  \sumstar_{\substack{\chi \bmod{q} \\ \chi^2 = \chi_0}}2^{\omega(q)} \sum_{m=1}^{\infty} \frac{\chi(m)}{\sqrt{m}} V_{\chi(-1)}\leg{m}{\sqrt{q}} w\left(\frac{q}{Q}\right).
\end{align*}

   Applying the above with Lemma \ref{lemma:quarticclass} again, we have $\mathcal{M} =\mathcal{M}^+ +\mathcal{M}^-$, with
\begin{align*}
 \mathcal{M}^{\pm} = \sumprime_{\substack {\mathcal{A}}}(1 \pm \chi_{\mathcal{A}}(-1)) \sum_{m=1}^{\infty} \frac{\chi_{\mathcal{A}}(m)}{\sqrt{m}}  V_{\chi_{\mathcal{A}}(-1)}\leg{m}{\sqrt{N(\mathcal{A})}} w\left(\frac{N(\mathcal{A})}{Q}\right),
\end{align*}
  where the dash on the sum over $\mathcal{A}$ indicates that the sum runs over square-free ideals of $\mathcal{O}_K$ that are co-prime to $2D_K$ and without rational prime divisor. \newline

It remains to evaluate $\mathcal{M}^{\pm}$. As the arguments are similar, we will only evaluate $\mathcal{M}^{+}$ in the sequel.  The results are summarized by
\begin{lemma}
\label{lemma1}
 We have
\begin{equation}
\label{eq:M1estimate}
 \mathcal{M}^{\pm} = QP^{\pm}_{K}(\log Q)+O(Q^{1-\delta_0 + \varepsilon}+Q^{3/4 + \varepsilon}),
\end{equation}
  where $\delta_0$ is given as in Theorem \ref{firstmoment}, $P^{\pm}_{K}(x)$ are given in \eqref{const} below.
\end{lemma}

\subsection{Evaluating $\mathcal{M}^+$, the main term}
\label{section:M1}
  We detect the condition that $\mathcal{A}$ has no rational prime divisor using the formula
\begin{equation*}
 \sum_{\substack{(d) |\mathcal{A} \\ d \in \mz }} \mu_{\mz}(d) =
\begin{cases}
 1, \quad \text{$\mathcal{A}$ has no rational prime divisor}, \\
 0, \quad \text{otherwise}.
\end{cases}
\end{equation*}
Here we define $\mu_{\mz}(d) = \mu(|d|)$, the usual M\"{o}bius function.  We apply this formula and change variables $\mathcal{A} \rightarrow d\mathcal{A}$ to the sum over $\mathcal{A}$.  Since $(d)$ is square-free as an ideal of $\mathcal{O}_K$, the condition that $d\mathcal{A}$ is square-free then simply means that $\mathcal{A}$ is square-free and $(d,\mathcal{A}) = 1$.  Thus
we have $\mathcal{M}^+=\mathcal{M}^+_1+\mathcal{M}^+_2$, where
\begin{align*}
 \mathcal{M}^+_1 &= \sum_{\substack{(d), \; d \in \mz \\ (d,2D_K)=1}} \mu_{\mz}(d) \sum_{m=1}^{\infty} \leg{m}{d}_K \frac{1}{\sqrt{m}} \sumstar_{\substack{\mathcal{A} \\ (\mathcal{A},2dD_K) = 1 }} \leg{m}{\mathcal{A}}_K V_1\left(\frac{m}{Q^{1/2}} \frac{Q^{1/2}}{\sqrt{N(d\mathcal{A})}} \right) w\left(\frac{N(d\mathcal{A})}{Q}\right), \\
 \mathcal{M}^+_2 &= \sum_{\substack{(d), \; d \in \mz \\ (d,2D_K)=1}} \mu_{\mz}(d) \sum_{m=1}^{\infty} \leg{-m}{d}_K \frac{1}{\sqrt{m}} \sumstar_{\substack{\mathcal{A} \\ (\mathcal{A},2dD_K) = 1 }} \leg{-m}{\mathcal{A}}_K V_1\left(\frac{m}{Q^{1/2}} \frac{Q^{1/2}}{\sqrt{N(d\mathcal{A})}} \right) w\left(\frac{N(d\mathcal{A})}{Q}\right),
\end{align*}
  where the asterisks indicate that $\mathcal{A}$ run over square-free ideals of $\mathcal{O}_K$.  \newline

  We evaluate $\mathcal{M}^+_{1}$ first. Using M\"{o}bius inversion to detect the condition that $\mathcal{A}$ is square-free, we get
\begin{equation*}
 \mathcal{M}^+_{1} = \sum_{\substack{(d), \; d \in \mz \\ (d,2D_K)=1}} \mu_{\mz}(d) \sum_{\substack{ \frak{l} \\ (\frak{l},2dD_K) = 1}} \mu_{K}(\frak{l}) \sum_{m=1}^{\infty} \leg{m}{d \frak{l}^2}_K \frac{1}{\sqrt{m}} \mathcal{M}_1(d,\frak{l},m),
\end{equation*}
where
\begin{equation*}
 \mathcal{M}_1(d,\frak{l},m) = \sum_{\substack{\mathcal{A} \\ (\mathcal{A},2dD_K) = 1}} \leg{m}{\mathcal{A}}_K V_1\left(\frac{m}{Q^{1/2}} \frac{Q^{1/2}}{\sqrt{N(d\frak{l}^2\mathcal{A})}} \right)  w\left(\frac{N(d\frak{l}^2\mathcal{A})}{Q}\right).
\end{equation*}
Next we use the Mellin transform of the weight function to express the sum over $\mathcal{A}$ as a contour integral involving the Hecke $L$-function.  By Mellin inversion,
\begin{equation*}
 V_1\left(\frac{m}{Q^{1/2}} \frac{Q^{1/2}}{\sqrt{N(d\frak{l}^2\mathcal{A})}} \right)  w\left(\frac{N(d\frak{l}^2\mathcal{A})}{Q}\right) = \frac{1}{2 \pi i} \int\limits_{(2)} \leg{Q}{N(d\frak{l}^2\mathcal{A})}^s \widetilde{f}(s) \dif s,
\end{equation*}
where
\begin{equation*}
\widetilde{f}(s) = \int\limits_0^{\infty} V_1\left(\frac{m}{Q^{1/2}} x^{-1/2} \right) w(x) x^{s-1} \dif x.
\end{equation*}
Integration by parts and using \eqref{2.07} shows $\widetilde{f}(s)$ is a function satisfying the bound for all $\text{Re}(s) >0$, $E\in \natn$,
\begin{equation} \label{ftildebound}
 \widetilde{f}(s) \ll \left( 1 + |s| \right)^{-E} \left( 1 + m/Q^{1/2} \right)^{-E}.
\end{equation}
Here the implied constant depends on $E$. \newline

   We then have
\begin{equation*}
 \mathcal{M}_1(d, \frak{l},m) = \frac{1}{2 \pi i} \int\limits_{(2)} \leg{Q}{N(d\frak{l}^2)}^s L(s, \psi_{4mD^2_Kd^2}) \widetilde{f}(s) \dif s,
\end{equation*}
   where
\begin{equation*}
 L(s, \psi_{4mD^2_Kd^2}) =\sum_{\substack{0 \neq \mathcal{A} \subset \mathcal{O}_K}}\psi_{4mD^2_Kd^2}(\mathcal{A}) N(\mathcal{A})^{-s},
\end{equation*}
   and
\begin{equation} \label{psidef}
   \psi_{4mD^2_Kd^2}(\mathcal{A}) := \begin{cases}
     \leg{4mD^2_Kd^2}{\mathcal{A}}_K \qquad & \text{when $(\mathcal{A}, 2D_K)=1$}, \\
     0 \qquad & \text{otherwise}.
    \end{cases}
\end{equation}
   Note that via \eqref{2.1} and the quadratic reciprocity law in $\mq$, there exists a positive integer $e$ independent of $m, d$ such that  $\psi_{4mD^2_Kd^2}((a))=1$ for any $a \in \mathcal{O}_K$ satisfying $a \equiv 1 \pmod {md^2(2D_K)^e}$. It follows from \cite[p. 470]{Newkirch} that $\psi_{4mD^2_Kd^2}$ can be regarded as a Hecke character of trivial infinite type $\pmod {md^2(2D_K)^e}$. \newline

We estimate $\mathcal{M}^+_{1}$  by moving the contour to the line with $\Re s =1/2$.  When $m$ is a square the Hecke $L$-function has a pole at $s=1$.  We set $\mathcal{M}_0$ to be the contribution to $\mathcal{M}^+_{1}$ of these residues, and $\mathcal{M}_1'$ to be the remainder. \newline

  We evaluate $\mathcal{M}_0$ first. Note that
\begin{equation*}
 \mathcal{M}_0 = \sum_{\substack{(d), d \in \mz \\ (d,2D_K)=1}} \mu_{\mz}(d) \sum_{\substack{ \frak{l} \\ (\frak{l},2dD_K) = 1}} \mu_{K}(\frak{l}) \sum_{m=1}^{\infty} \leg{m}{d \frak{l}^2}_K\frac{1}{\sqrt{m}} \frac{Q}{N(d\frak{l}^2)} \widetilde{f}(1) \text{Res}_{s=1} L(s, \psi_{4mD^2_Kd^2}),
\end{equation*}
where using the Mellin inversion formula yields
\begin{equation*}
 \widetilde{f}(1) = \int\limits_0^{\infty}  V_1\left(\frac{m}{Q^{1/2}} x^{-1/2} \right) w(x) \dif x = \frac{1}{2 \pi i} \int\limits_{(2)} \leg{Q^{1/2}}{m}^s \widetilde{w} \left( 1+\frac s2 \right) \gamma_1(s) \frac {\dif s}{s},
\end{equation*}	
  with $\gamma_1(s)$ defined in \eqref{vdef} and
\begin{align*}
\widetilde{w}(s) = \int\limits_0^{\infty} w(x) x^{s-1} \dif x.
\end{align*}

From our discussions above, it is not difficult to see that $\psi_{4mD^2_Kd^2}$ is the principal character only if $m$ is a square, in which case
\begin{equation*}
 L(s, \psi_{4mD^2_Kd^2}) = \zeta_{K}(s) \prod_{\mathfrak{p} | 2dmD_K} \left( 1 - N(\mathfrak{p})^{-s} \right).
\end{equation*}
  Here and in what follows, we use $\mathfrak{p}$ or $\mathfrak{p}_i$ to denote prime ideals in $\mathcal{O}_K$. \newline

Let $C_{K,1}$ be the residue of $\zeta_{K}(s)$ at $s=1$, then
\begin{equation*}
 \mathcal{M}_0 =  C_{K,1} Q \sum_{m=1}^{\infty} \frac{\widetilde{f}(1)}{m} \prod_{\mathfrak{p} | 2mD_K} \left( 1- N(\mathfrak{p})^{-1} \right)  \sum_{\substack{(d), \; d \in \mz \\ (d,2mD_K) =1}} \frac{\mu_{\mz}(d)}{d^2}\prod_{\mathfrak{p} | d} \left( 1 - N(\mathfrak{p})^{-1} \right) \sum_{\substack{\frak{l} \\ (\frak{l},2mdD_K)=1}} \frac{\mu_{K}(\frak{l})}{N(\frak{l}^2)}.
\end{equation*}

   Computing the sum over $l$ explicitly, we obtain
\begin{align*}
 \mathcal{M}_0 &= C_{K,1} \zeta^{-1}_{K}(2) Q \sum_{m=1}^{\infty} \frac{\widetilde{f}(1)}{m} \prod_{\mathfrak{p} | 2mD_K} \left( 1- N(\mathfrak{p})^{-1} \right)  \sum_{\substack{(d), \; d \in \mz \\ (d,2mD_K) =1}} \frac{\mu_{\mz}(d)}{d^2}\prod_{\mathfrak{p} | d} \left( 1 - N(\mathfrak{p})^{-1} \right) \prod_{\mathfrak{p}| 2mdD_K} (1- N(\mathfrak{p})^{-2})^{-1} \\
 &= C_{K,1} \zeta^{-1}_{K}(2) Q \sum_{m=1}^{\infty} \frac{\widetilde{f}(1)}{m}\prod_{\mathfrak{p} | 2mD_K} \left( 1+ N(\mathfrak{p})^{-1} \right)^{-1}  \sum_{\substack{(d) , \; d \in \mz \\ (d,2mD_K) =1 }} \frac{\mu_{\mz}(d)}{d^2}\prod_{\mathfrak{p} | d} \left( 1 + N(\mathfrak{p} )^{-1} \right)^{-1}.
\end{align*}

   We define
\begin{align*}
  C_{K,2}= \prod_{\mathfrak{p} | 2D_K} \left( 1+ N(\mathfrak{p} \right)^{-1})^{-1}\sum_{\substack{(d), \; d \in \mz \\ (d,2D_K)=1}} \frac{\mu_{\mz}(d)}{d^2}\prod_{\mathfrak{p} | d} \left( 1 + N(\mathfrak{p})^{-1} \right)^{-1}.
\end{align*}

   It is clear that $C_{K,2}$ is a constant. Using this and setting $\tilde{m} = m/(m,2D_k)$, we have
\begin{align*}
 \mathcal{M}_0 = C_{K,1}C_{K,2} \zeta^{-1}_{K}(2) Q \sum_{m=1}^{\infty} \frac{\widetilde{f}(1)}{m} \prod_{\mathfrak{p}_1 | \tilde{m}} \left( 1+ N(\mathfrak{p}_1 )^{-1} \right)^{-1}
\prod_{p | \tilde{m}} \left( 1-p^{-2}\prod_{\mathfrak{p}_2| p} \left( 1 + N(\mathfrak{p}_2)^{-1} \right)^{-1} \right)^{-1},
\end{align*}
where $p$ runs over rational primes.  Let
\begin{equation*}
 C_K(s) =\zeta^{-1}(s) \sum_{m=1}^{\infty} m^{-s}  \prod_{\mathfrak{p}_1 | \tilde{m}} \left( 1+ N(\mathfrak{p}_1)^{-1} \right)^{-1}
\prod_{p | \tilde{m}} \left( 1-p^{-2}\prod_{\mathfrak{p}_2| p} (1 + N(\mathfrak{p}_2)^{-1})^{-1} \right)^{-1},
\end{equation*}
  where $\zeta(s)$ is the Riemann zeta-function. Expressing $C_K(s)$ as an Euler product, one checks easily that $C_K(s)$ is holomorphic, converges absolutely for $\text{Re}(s) \geq 1/2 + \delta > 1$ and can be analytically continued to $\text{Re}(s) > 1/2$.  Then
\begin{equation*}
 \mathcal{M}_0 = C_{K,1}C_{K,2}\zeta^{-1}_{K}(2)  Q \frac{1}{2 \pi i} \int\limits_{(2)} Q^{s/2} C_K(1 + s)\zeta (1+s)\widetilde{w} \left( 1+\frac s2 \right) \gamma_1(s) \frac {\dif s}{s}.
\end{equation*}

  We move the contour of integration to $-1/2 + \varepsilon$, crossing a pole of order $2$ at $s=0$ only.  The new contour contributes $O(Q^{3/4 + \varepsilon})$, by noting that we have $\zeta(1+s) \ll |1+s|$ on this line (see (3)  on p. 79 of \cite{Da}). Using the fact that the Laurent expansion of $\zeta(s)$ at $s=1$
  has the form \cite[Corollary 1.16]{M&V}
\begin{equation*}
  \zeta(s)=\frac 1{s-1}+\gamma_0+\sum^{\infty}_{k=1}a_k(s-1)^k, \quad a_k \in \mc,
\end{equation*}
  where $\gamma_0$ is the Euler constant.
  The pole at $s=0$ gives $QP^+_{K}(\log Q)$, where we define
\begin{align}
\label{const}
   P_K(x) & = P^+_{K}(x)+ P^-_{K}(x), \; \mbox{with} \\
   P^+_{K}(x) &=C_{K,1}C_{K,2}\zeta^{-1}_{K}(2)\left(  \frac 12C_K(1)\widetilde{w}(1)x+C'_K(1)+\frac {\widetilde{w}'(1)}{2}+ \gamma'_1(0)+\gamma_0\left(C_K(1)+\widetilde{w}(1)\right) \right), \nonumber \\
   P^-_{K}(x) &=C_{K,1}C_{K,2}\zeta^{-1}_{K}(2)\left(  \frac 12C_K(1)\widetilde{w}(1)x+C'_K(1)+\frac {\widetilde{w}'(1)}{2}+ \gamma'_{-1}(0)+\gamma_0\left(C_K(1)+\widetilde{w}(1)\right) \right).  \nonumber
\end{align}

We then conclude that
\begin{align}
\label{m0}
 \mathcal{M}_0 = QP^+_{K}(\log Q)+O \left( Q^{3/4 + \varepsilon} \right).
\end{align}

\subsection{Evaluating $\mathcal{M}'_1$ and $\mathcal{M}^+_{2}$}
\label{section:remainderterm}

In this section, we estimate $\mathcal{M}'_1$ and $\mathcal{M}^+_{2}$.   Recalled that $\mathcal{M}_1'$ is $\mathcal{M}_1^+-\mathcal{M}_0$.  More preicisely,
\[ \mathcal{M}_1' = \frac{1}{2\pi i}  \sum_{\substack{(d), \; d \in \mz \\ (d,2D_K)=1}} \mu_{\mz}(d) \sum_{\substack{ \frak{l} \\ (\frak{l},2dD_K) = 1}} \mu_{K}(\frak{l}) \sum_{m=1}^{\infty} \leg{m}{d \frak{l}^2}_K \frac{1}{\sqrt{m}} \int\limits_{(1/2)} \leg{Q}{N(d\frak{l}^2)}^s L(s, \psi_{4mD^2_Kd^2}) \widetilde{f}(s) \dif s . \]

 By bounding everything with absolute values and using \eqref{ftildebound} to bound $\widetilde{f}$, we see that, for some large $E \in \natn$,
\begin{equation}
\label{M1'}
|\mathcal{M}_1'| \ll  \sum_{d \leq c_1 \sqrt{Q}} \sum_{N(\frak{l}) \leq c_2 \sqrt{Q}} \frac{1}{\sqrt{N(d\frak{l}^2)}} \sum_{m} \frac{\sqrt{Q}}{\sqrt{m}} \left( 1+m/Q^{1/2} \right)^{-E} \int\limits_0^{\infty} \left| L(1/2 + it, \psi_{4mD^2_Kd^2}) \right| (1+|t|)^{-E} \dif t.
\end{equation}
Here $c_1$ and $c_2$ are constants, chosen according to the size of the support the weight function $w$.  In view of the factor $(1+m/Q^{1/2})^{-E}$, we may truncate the sum over $m$ above to $m \leq M \ll Q^{1/2+\varepsilon}$ for $\varepsilon>0$ with a small error.   Indeed, if $m \gg Q^{1/2+\varepsilon}$, then
\[ \left( 1+ \frac{m}{Q^{1/2}} \right)^{-1} \ll \min \left\{ \frac{Q^{1/2}}{m} , \; \frac{1}{Q^{\varepsilon}} \right\} . \]
So taking $E$ large so that $\varepsilon E >10$, using the convexity found for the $L$-function and summing up all the other variables trivially, we see that the contribution to $M_1'$ from $m \gg Q^{1/2+\varepsilon}$ is
\[ \ll Q^{-5}  \]
which is negligible. \newline

To estimate the contribution from the small $m$'s, we need a better bound for $L \left( 1/2 + it, \psi_{4mD^2_Kd^2} \right)$.  The character $\psi_{4mD^2_Kd^2}$ is induced by a primitive character $\psi'$ with conductor $f$ satisfying $f | m(2D_K)^2$.  Note that from its definition in \eqref{psidef}, $\psi_{4mD^2_Kd^2}$ is induced by a character, not necessarily primitive, modulo $4mD_K^2$.  So the conductor $f$ here is independent of $d$.  It follows that for any $\varepsilon>0$,
\begin{align}
\label{Lbound0}
 L\left( 1/2 + it, \psi_{4mD^2_Kd^2} \right) \ll (md)^{\varepsilon}|L(1/2 + it, \psi')|,
\end{align}
where the implied constant here depends on $\varepsilon$ \newline

     Now, the Hecke $L$-function $L(s, \psi')$, viewed as a degree two $L$-function over $\mq$, has analytic conductor $\ll N(m) (1+ t^2)$ (see \cite[Theorem 12.5]{Iwaniec}). It follows from a result on the subconvexity bound for degree two $L$-functions over any fixed number field by P. Michel and A. Venkatesh \cite{Michel&Venkatesh} that we have an absolute constant $\delta_0>0$, independent of the number field, such that
\begin{align} \label{Lbound}
 L \left( 1/2 + it, \psi' \right) \ll (N(m)(1+t^2))^{1/4-\delta_0}.
\end{align}
Applying \eqref{Lbound0} and \eqref{Lbound}, we deduce the following estimation to bound the sum over $m$:
\begin{equation*}
 \sum_{m \leq M} \frac{1}{\sqrt{m}} \left| L \left(1/2 + it, \psi_{4mD^2_Kd^2} \right) \right| \ll d^{\epsilon} M^{1-2\delta_0+\varepsilon} (1+t^2)^{1/4-\delta_0}.
\end{equation*}

   We sum trivially over $d$ and $\frak{l}$ in \eqref{M1'} to see that
\begin{equation}
\label{eq:M1'}
|\mathcal{M}_1'| \ll Q^{1-\delta_0 + \varepsilon}.
\end{equation}

Using similar arguments, we obtain the same estimation for $\mathcal{M}^+_{2}$ as above.  Combining \eqref{eq:M1'} with \eqref{m0} gives \eqref{eq:M1estimate}.

\subsection{Conclusion}  As one readily deduces \eqref{eq:1} from Lemma \ref{lemma1}, this completes the proof of Theorem \ref{firstmoment}. \newline

\noindent{\bf Acknowledgments.} P. G. is supported in part by NSFC grant 11871082 and L. Z. by the FRG grant PS43707 and the Silverstar Fund PS49334 at UNSW.

\bibliography{biblio}
\bibliographystyle{amsxport}

\vspace*{.5cm}

\noindent\begin{tabular}{p{8cm}p{8cm}}
School of Mathematics and Systems Science & School of Mathematics and Statistics \\
Beihang University & University of New South Wales \\
Beijing 100191 China & Sydney NSW 2052 Australia \\
Email: {\tt penggao@buaa.edu.cn} & Email: {\tt l.zhao@unsw.edu.au} \\
\end{tabular}

\end{document}